\documentclass[12pt,reqno]{amsart}

\usepackage{amssymb,amsthm}
\usepackage{amsmath}
\usepackage{setspace}
\usepackage{dcpic,pictexwd}
\usepackage[all]{xy}
\usepackage[pdftex]{graphicx}
\usepackage{xcolor}
\usepackage{mathrsfs}
\usepackage[pdftex]{hyperref}
\usepackage{bbm}

\newtheorem{theorem}{Theorem}[section]
\newtheorem{lemma}[theorem]{Lemma}
\newtheorem{corollary}[theorem]{Corollary}
\newtheorem{claim}[theorem]{Claim}

\theoremstyle{definition}

\theoremstyle{remark}
\newtheorem{remark}[theorem]{Remark}

\numberwithin{equation}{section}

\newcommand{\Ext}{\mathrm{Ext}}
\newcommand{\End}{\mathrm{End}}

\newcommand{\Hom}{\mathrm{Hom}}

\newcommand{\Ca}{\mathcal{C}}
\newcommand{\Fun}{\mathrm{F}}

\newcommand{\Def}{\mathrm{Def}}
\newcommand{\Sets}{\mathrm{Sets}}

\newcommand{\Ob}{\mathrm{Ob}}

\newcommand{\Z}{\mathbb{Z}}
\newcommand{\SEnd}{\underline{\End}}
\newcommand{\A}{\Lambda}

\newcommand{\m}{\mathfrak{m}}
\renewcommand{\k}{\Bbbk}

\newcommand{\invlim}{\varprojlim}

\begin{document}

\title{On universal deformation rings and stable homogeneous tubes}



\author[Caranguay-Mainguez]{Jhony F. Caranguay-Mainguez}
\address{Instituto de Matem\'aticas, Universidad de Antioquia, Medell\'{\i}n. Colombia}
\curraddr{}
\email{jhony.caranguay@udea.edu.co}
\thanks{}

\author[Rizzo]{Pedro Rizzo}
\address{Instituto de Matem\'aticas, Universidad de Antioquia, Medell\'{\i}n. Colombia}
\curraddr{}
\email{pedro.hernandez@udea.edu.co}
\thanks{}

\author[V\'elez-Marulanda]{Jos\'e A. V\'elez-Marulanda}
\address{Department of Applied Mathematics \& Physics, Valdosta State University, GA, USA}
\address{Facultad de Matem\'aticas e Ingenier\'{\i}as, Fundaci\'on Universitaria Konrad Lorenz, Bogot\'a, Colombia}

\curraddr{}
\email{javelezmarulanda@valdosta.edu}
\thanks{}


\date{}

\dedicatory{}

\commby{}

\begin{abstract}
Let $\k$ be a field of any characteristic and let $\A$ be a finite dimensional $\k$-algebra. We prove that if $V$ is a finite dimensional right $\A$-module that lies in the mouth of a stable homogeneous tube $\mathfrak{T}$ of the Auslander-Reiten quiver $\A$ with $\SEnd_\A(V)$ a division ring, then $V$ has a versal deformation ring $R(\A,V)$ isomorphic to $\k[\![t]\!]$. As consequence we obtain that if $\k$ is algebraically closed, $\A$ is a symmetric special biserial $\k$-algebra and $V$ is a band $\A$-module with $\SEnd_\A(V) \cong \k$ that lies in the mouth of its homogeneous tube, then $R(\A,V)$ is universal and isomorphic to $\k[\![t]\!]$.     
\end{abstract}

\maketitle
\renewcommand{\labelenumi}{\textup{(\roman{enumi})}}
\renewcommand{\labelenumii}{\textup{(\roman{enumi}.\alph{enumii})}}

\section{Introduction}\label{sec1}

Throughout this article, we assume that $\k$ is a field of any characteristic. We denote by $\widehat{\Ca}$ the category of all complete local commutative Noetherian $\k$-algebras with residue field $\k$. In particular, the morphisms in $\widehat{\Ca}$ are continuous $\k$-algebra homomorphisms that induce the identity map on $\k$. Let $\A$ be a finite dimensional $\k$-algebra and let $V$ be a finite dimensional right $\A$-module. It was proved in \cite[Prop. 2.1]{blehervelez} that  $V$ has a well-defined versal deformation ring $R(\A,V)$ which is an object in $\widehat{\Ca}$ and which is universal provided that  $\End_\A(V)$, the endomorphism ring of $V$, is isomorphic to $\k$. Moreover, it follows from  \cite[Prop. 2.5]{blehervelez} that versal deformation rings are preserved under Morita equivalences. If we assume further that $V$ is a non-projective indecomposable Gorenstein-projective right $\A$-module whose stable endomorphism ring $\SEnd_\A(V)$ is isomorphic to $\k$, then it follows from \cite[Thm. 1.2 (ii) ]{bekkert-giraldo-velez} and \cite[Thm. 1.2 (i)]{velez4} that $R(\A,V)$ is also universal and stable under syzygies.  Moreover, it follows from \cite[Thm. 1.2 (ii)]{velez4} that under this situation, $R(\A, V)$ is preserved under singular equivalences of Morita type with level (in the sense of \cite{wang}) between Gorenstein $\k$-algebras. These results have been used in \cite{bekkert-giraldo-velez, velez2,velez3,velez4} to classify universal deformation rings of finite dimensional Gorenstein-proyective modules over gentle, skew-gentle and Nakayama algebras with  all these algebras of infinite global dimension. On the other hand, the aforementioned results also apply to self-injective $\k$-algebras and stable equivalences of Morita type (in the sense of \cite{broue}). An important class of self-injective $\k$-algebras are the symmetric special biserial (see e.g. \cite[\S 2.3]{schroll}). It follows from the results in \cite{buri} that if $\A$ is a symmetric special biserial, then every non-projective indecomposable right $\A$-module $V$ is either a string or a band module. Under this situation, it follows from e.g. \cite[\S II.4.3]{erdmann} that the band modules lie in stable homogeneous tubes i.e. components of the stable Auslander-Reiten quiver of the form $\Z\mathbb{A}_\infty/\langle \tau_\A\rangle$, where $\tau_\A$ is the Auslander-Reiten translation (see \cite[\S 3.1]{ringel2} and \cite[\S X.1]{simson} and Figure \ref{fig1}). We refer to the reader to \cite[\S 2.2]{schroll2} for a description of such string and band modules.  Assume for now that $\k$ is algebraically closed, that $\A$ is a symmetric special biserial algebra and $V$ is a non-projective right $\A$-module such that $\SEnd_\A(V) \cong \k$. If $V$ is a string $\A$-module, then $R(\A,V)$ has been determined under many cases for $\A$ (see e.g. \cite{bleher9,blehervelez, bleher15, calderon-giraldo-rueda-velez, meyer, velez}). If $V$ is a band module, then $R(\A,V)$ has been determined e.g when $\A$ is a certain algebra of dihedral type (see \cite[Thm. 1.2 (iii)]{blehervelez}), when $\A$ is of dihedral type of polynomial growth (see \cite[Thm. 1.1 (iii)]{bleher9}) and for generalized Brauer tree algebras of polynomial growth (see \cite[Prop. 5.3]{meyer}). In all these cases, turned out that $V$ belongs to the mouth of its stable homogeneous tube and that $R(\A,V)\cong \k[\![t]\!]$. 

Our aim is to prove the following result. 

\begin{theorem}\label{thm1}
Let $\A$ be a finite dimensional $\k$-algebra where $\k$ is a field of arbitrary characteristic. Assume that $V= V[1]$ is an indecomposable finite dimensional right $\A$-module such that $\SEnd_\A(V)$ is a division ring and which lies in the mouth of a stable homogeneous tube $\mathfrak{T}$ of the Auslander-Reiten quiver of $\A$ (as in Figure \ref{fig1}) and such that for all $\ell \geq 2$, there exists a short exact sequence of right $\A$-modules
\begin{equation}\label{seslong}
0\to V[\ell-1]\to V[\ell]\to V[1]\to 0. 
\end{equation}
Then the versal deformation ring $R(\A,V)$ is isomorphic to $\k[\![t]\!]$. 
\end{theorem}

It follows from \cite[\S 2.2 Theorem (a)]{simson} that if $\End_\A(V)\cong \k$ with $\k$ algebraically closed and $\Ext_\A^2(V,V) = 0$ (which is always true when $\A$ is hereditary), then there exists such short exact sequence as in (\ref{seslong}). By \cite[Prop. 2.1]{blehervelez} and Theorem \ref{thm1}, in this situation we obtain that $R(\A,V)$ is universal and isomorphic to  $\k[\![t]\!]$. More precisely, recall from \cite[\S X]{simson} that an indecomposable right $\A$-module $V$ is {\it regular} if it belongs neither to the post-projective component $\mathcal{P}(\A)$ nor to the pre-injective component $\mathcal{Q}(\A)$ of the Auslander-Reiten quiver of $\A$. Following \cite[\S XI.2.5 Definition]{simson}, $V$ is {\it simple regular} if $V$ is a regular right $\A$-module having no proper regular submodules. It follows from the arguments in the proof of \cite[\S XI.2.8 Theorem]{simson} that if $V$ is a simple regular module, then $V$ has endomorphism ring isomorphic to $\k$ and that $V$ lies in the mouth of its stable tube. The following result is again a direct consequence of \cite[Prop. 2.5]{blehervelez}, Theorem \ref{thm1} and \cite[\S XIII.2.24]{simson}. 
 
\begin{corollary}\label{cor2}
Assume that $\k$ is algebraically closed and that $\A$ is the path algebra $\k\Delta$, where $\Delta$ is one of the canonical oriented quivers $\Delta(\widetilde{\mathbb{A}}_{p,q})$, with $p,q\geq 1$, $\Delta(\widetilde{\mathbb{D}}_{m})$ with $m\geq 4$, and $\Delta(\widetilde{\mathbb{E}}_\ell)$ with $\ell = 6, 7, 8$ as in Figure \ref{fig2}. Then for all $\lambda \in \k^\ast$ there exists an homogeneous tube $\mathfrak{T}_\lambda$ of the Auslander-Reiten quiver of $\A$ such that it contains a simple regular $\k\Delta$-module $E^{(\lambda)}$ that lies in its mouth and which has endomorphism ring isomorphic to $\k$.  In this situation, the versal deformation ring $R(\A, E^{(\lambda)})$ is universal and isomorphic to $\k[\![t]\!]$.  
\end{corollary}

On the other hand, let us still assume that $\k$ is algebraically closed and that $\A$ is a symmetric special biserial $\k$-algebra. If $V$ is a band module whose stable endomorphism ring is isomorphic to $\k$, then $V \cong \tau_\A V\cong \Omega^2 V$, where $\Omega V$ is the kernel of a projective cover $P\to V$, which is unique up to isomorphism. This implies that $\Ext_\A^2(V,V) \cong \SEnd_\A(V) \cong \k$. This implies that the results from \cite[\S 2.2 Theorem]{simson} cannot be applied. However, using e.g. the aformentioned description of band modules from \cite[\S 2.3]{schroll}, if $V=V[1]$ is a band module that lies in the mouth of its stable homogeneous tube of the stable Auslander-Reiten quiver of $\A$, then we still have a short exact sequence of the form (\ref{seslong}) (see \S \ref{subsec2}). Thus we obtain the following result.  

\begin{corollary}\label{cor1}
Assume that $\k$ is algebraically closed and $\A$ is a symmetric special biserial algebra. If $V$ is a band right $\A$-module that lies in the mouth of its stable homogenous tube of the stable Auslander-Reiten quiver of $\A$ and $\SEnd_\A(V) \cong \k$, then the versal deformation ring $R(\A,V)$ is universal and isomorphic to $\k[\![t]\!]$.  
\end{corollary} 

Observe that Corollary \ref{cor1} generalizes the aforementioned results in \cite{blehervelez,bleher15} and \cite{meyer} where the arguments use elementary matrices that are defined according to the form of the band and the scalar $\lambda \in \k^\ast$ that define the corresponding band module to prove these results. This argument is clearly avoided by using Corollary \ref{cor1}. Note as well that the path algebra $\k\Delta(\widetilde{\mathbb{E}}_\ell)$ with $\ell = 6, 7, 8$ is not special biserial and thus the techniques from \cite{blehervelez,bleher15} and \cite{meyer} cannot be immediately used in this situation. Next assume that $\A$ is a $\k$-algebra with finite global dimension. Then every finite dimensional Gorenstein-projective right $\A$-module $V$ is projective and thus it has a universal deformation ring $R(\A,V)$ isomorphic to $\k$ (see Remark \ref{rem0} (ii)). It follows from Corollary \ref{cor2} that the simple regular modules $E^{(\lambda)}$ provide examples of non-Gorenstein-proyective modules over algebras of finite global dimension that have a non-trivial deformation ring. 

\begin{figure}
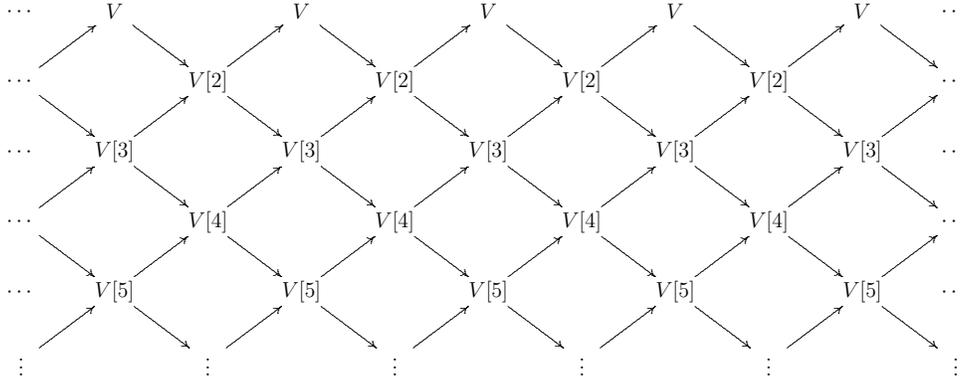

\scalebox{0.68}{
$
\begindc{\commdiag}[130]
\obj(-20,7)[2]{$\cdots$}
\obj(-16,7)[3]{$V$}
\obj(-8,7)[4]{$V$}
\obj(0,7)[5]{$V$}
\obj(8,7)[6]{$V$}
\obj(16,7)[7a]{$V$}
\obj(20,7)[7b]{$\cdots$}

\obj(-20,4)[7c]{$\cdots$}
\obj(-12,4)[8]{$V[2]$}
\obj(-4,4)[9]{$V[2]$}
\obj(4,4)[10]{$V[2]$}
\obj(12,4)[11]{$V[2]$}
\obj(20,4)[11a]{$\cdots$}
\mor{11}{7a}{}
\mor{7a}{11a}{}
\mor{3}{8}{}
\mor{7c}{3}{}
\obj(-20,1)[12a]{$\cdots$}
\obj(-16,1)[12]{$V[3]$}
\obj(-8,1)[13]{$V[3]$}
\obj(0,1)[14]{$V[3]$}
\obj(8,1)[15]{$V[3]$}
\obj(16,1)[16]{$V[3]$}
\obj(20,1)[16a]{$\cdots$}

\mor{7c}{12}{}
\mor{12}{8}{}
\mor{11}{16}{}
\mor{16}{11a}{}
\obj(-20,-2)[17a]{$\cdots$}
\obj(-12,-2)[17]{$V[4]$}
\obj(-4,-2)[18]{$V[4]$}
\obj(4,-2)[19]{$V[4]$}
\obj(12,-2)[20]{$V[4]$}
\obj(20,-2)[20a]{$\cdots$}
\mor{17a}{12}{}
\mor{12}{17}{}
\mor{20}{16}{}
\mor{16}{20a}{}


\obj(-20,-5)[21a]{$\cdots$}
\obj(-16,-5)[21]{$V[5]$}
\obj(-8,-5)[22]{$V[5]$}
\obj(0,-5)[23]{$V[5]$}
\obj(8,-5)[24]{$V[5]$}
\obj(16,-5)[25]{$V[5]$}
\obj(20,-5)[25a]{$\cdots$}
\mor{17a}{21}{}
\mor{21}{17}{}
\mor{20}{25}{}
\mor{25}{20a}{}
\obj(-20,-8)[26a]{$\vdots$}
\obj(-12,-8)[26]{$\vdots$}
\obj(-4,-8)[27]{$\vdots$}
\obj(4,-8)[28]{$\vdots$}
\obj(12,-8)[29]{$\vdots$}
\obj(20,-8)[29a]{$\vdots$}
\mor{26a}{21}{}
\mor{21}{26}{}
\mor{29}{25}{}
\mor{25}{29a}{}



\mor{8}{4}{}
\mor{4}{9}{}
\mor{9}{5}{}
\mor{5}{10}{}
\mor{10}{6}{}
\mor{6}{11}{}
\mor{8}{13}{}
\mor{13}{9}{}
\mor{9}{14}{}
\mor{14}{10}{}
\mor{10}{15}{}
\mor{15}{11}{}
\mor{17}{13}{}
\mor{13}{18}{}
\mor{18}{14}{}
\mor{14}{19}{}
\mor{19}{15}{}
\mor{15}{20}{}
\mor{17}{22}{}
\mor{22}{18}{}
\mor{18}{23}{}
\mor{23}{19}{}
\mor{19}{24}{}
\mor{24}{20}{}
\mor{26}{22}{}
\mor{22}{27}{}
\mor{27}{23}{}
\mor{23}{28}{}
\mor{28}{24}{}
\mor{24}{29}{}
\enddc
$
}
\caption{The stable homogeneous tube $\mathfrak{T}$ containing the right $\A$-module $V$.}\label{fig1}
\end{figure}

\begin{figure}
\begin{align*}
&\xymatrix@=10pt{&&\ar[dl]\underset{1}{\bullet}&\ar[l]\underset{2}{\bullet}&\ar[l]\cdots&\ar[l]\underset{p-1}{\bullet}\\\Delta(\widetilde{\mathbb{A}}_{p,q}):&\underset{0}{\bullet}&&&&&\underset{p+q-1}{\bullet}\ar[dl]\ar[ul]\\&&\ar[ul]\underset{p}{\bullet}&\ar[l]\underset{p+1}{\bullet}&\ar[l]\cdots&\ar[l]\underset{p+q-2}{\bullet}}\\
&\xymatrix@=10pt{&\underset{1}{\bullet}&&&&&\ar[dl]\underset{m}{\bullet}\\\Delta(\widetilde{\mathbb{D}}_m):&&\ar[ul]\ar[dl]\underset{3}{\bullet}&\ar[l]\underset{4}{\bullet}&\ar[l]\cdots&\ar[l]\underset{m-1}{\bullet}&\\&\underset{2}{\bullet}&&&&&\ar[ul]\underset{m+1}{\bullet}}\\
&\xymatrix@=10pt{&&&\underset{5}{\bullet}\ar[d]&&&\\\Delta(\widetilde{\mathbb{E}}_6):&&&\underset{4}{\bullet}\ar[d]&&&\\&\underset{3}{\bullet}\ar[r]&\underset{2}{\bullet}\ar[r]&\underset{1}{\bullet}&\underset{6}{\bullet}\ar[l]&\underset{7}{\bullet}\ar[l]&}\\
&\xymatrix@=10pt{\Delta(\widetilde{\mathbb{E}}_7):&&&&\underset{5}{\bullet}\ar[d]&&&\\&\underset{4}{\bullet}\ar[r]&\underset{3}{\bullet}\ar[r]&\underset{2}{\bullet}\ar[r]&\underset{1}{\bullet}&\underset{6}{\bullet}\ar[l]&\underset{7}{\bullet}\ar[l]&\underset{8}{\bullet}\ar[l]}\\
&\xymatrix@=10pt{\Delta(\widetilde{\mathbb{E}}_8):&&&&\underset{5}{\bullet}\ar[d]&&&\\&\underset{4}{\bullet}\ar[r]&\underset{3}{\bullet}\ar[r]&\underset{2}{\bullet}\ar[r]&\underset{1}{\bullet}&\underset{6}{\bullet}\ar[l]&\underset{7}{\bullet}\ar[l]&\underset{8}{\bullet}\ar[l]&\underset{9}{\bullet}\ar[l]}
\end{align*}
\caption{The canonical oriented Euclidean quivers.}\label{fig2}
\end{figure}
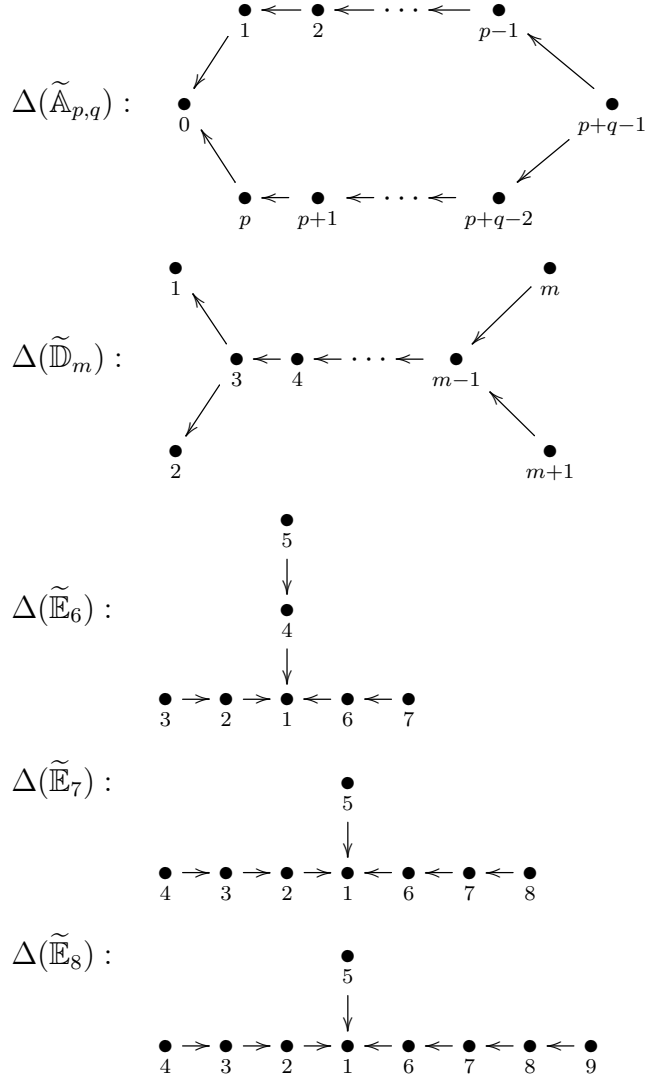

Corollary \ref{cor1} has an immediate implication for a particular class of algebras, specifically those defined by their representation type.

\begin{corollary}\label{cor2}
Let $\k$ be an algebraically closed field. If $\A$ is a symmetric special biserial algebra that is $\tau$-tilting infinite, then there exists at least a band $b$ such that for all $\lambda\in\k^{\ast}$ the band right $\Lambda$-module $V(b,\lambda, 1)$ has a universal deformation ring $R(\A,V)$ isomorphic to $\k[\![t]\!]$.
\end{corollary}

A key characterization established in \cite[Thm. 5.1]{schtreval} states that a symmetric special biserial algebra is $\tau$-tilting infinite precisely when there exists a least a band that induces a band module that is a brick. This fundamental result, in turn, implies the special property observed in Corollary \ref{cor2}.

\section{Proof of the main result}\label{sec2}
Throughout this section we keep the notation introduced in \S \ref{sec1}. 
Let $V$ be a right $\A$-module and let $R$ be a fixed but arbitrary object in $\widehat{\Ca}$. Following \cite{blehervelez}, a {\it lift} $(M,\phi)$ 
of $V$ over $R$ is a finitely generated right $R\A$-module $M$ 
that is free over $R$ 
together with an isomorphism of $\A$-modules $\phi:\k\otimes_RM\to V$, where $R\A:=R\otimes_{\k}\A$. Two lifts $(M,\phi)$ and $(M',\phi')$ over $R$ are {\it isomorphic} 
if there exists an $R\A$-module 
isomorphism $f:M\to M'$ such that $\phi'\circ (\mathrm{id}_\k\otimes_R f)=\phi$.
If $(M,\phi)$ is a lift of $V$ over $R$, we  denote by $[M,\phi]$ its isomorphism class and say that $[M,\phi]$ is a {\it deformation} of $V$ 
over $R$. We denote by $\Def_\A(V,R)$ the 
set of all deformations of $V$ over $R$. The {\it deformation functor} corresponding to $V$ is the 
covariant functor $\widehat{\Fun}_V:\widehat{\Ca}\to \Sets$ defined as follows: for all objects $R$ in $\widehat{\Ca}$, define $\widehat{\Fun}_V(R)=\Def_
\A(V,R)$, and for all morphisms $\theta:R\to 
R'$ in $\widehat{\Ca}$, 
let $\widehat{\Fun}_V(\theta):\Def_\A(V,R)\to \Def_\A(V,R')$ be defined as $\widehat{\Fun}_V(\theta)([M,\phi])=[R'\otimes_{R,\theta}M,\phi_\theta]$, 
where $\phi_\theta: \k\otimes_{R'}
(R'\otimes_{R,\theta}M)\to V$ is the composition of $\A$-module isomorphisms 
\[\k\otimes_{R'}(R'\otimes_{R,\theta}M)\cong \k\otimes_RM\xrightarrow{\phi} V.\] 


Suppose there exists an object $R(\A,V)$ in $\widehat{\Ca}$ and a deformation $[U(\A,V), \phi_{U(\A,V)}]$ of $V$ over $R(\A,V)$ with the 
following property. For all objects $R$ in $\widehat{\Ca}$ and for all deformations $[M,\phi]$ of $V$ over $R$, there exists a morphism $\psi_{R(\A,V),R,[M,\phi]}:R(\A,V)\to R$ 
in $\widehat{\Ca}$ such that 
\[\widehat{\Fun}_V(\psi_{R(\A,V),R,[M,\phi]})[U(\A,V), \phi_{U(\A,V)}]=[M,\phi],\]
and moreover, $\psi_{R(\A,V),R,[M,\phi]}$ is unique if $R$ is the ring of dual numbers $\k[\epsilon]$ with $\epsilon^2=0$.  Then $R(\A,V)$ and $
[U(\A,V),\phi_{U(\A,V)}]$ are called the {\it versal deformation ring} and {\it versal deformation} of $V$, respectively. If the morphism $
\psi_{R(\A,V),R,[M,\phi]}$ is unique for all $R\in\Ob(\widehat{\Ca})$ and deformations $[M,\phi]$ of $V$ over $R$, then $R(\A,V)$ and $[U(\A,V),\phi_{U(\A,V)}]$ are 
called the {\it universal deformation ring} and the {\it universal deformation} of $V$, respectively.  In other words, the universal deformation 
ring $R(\A,V)$ represents the deformation functor $\widehat{\Fun}_V$ in the sense that $\widehat{\Fun}_V$ is naturally isomorphic to the $\Hom$ 
functor $\Hom_{\widehat{\Ca}}(R(\A,V),-)$. We denote by $\Fun_V$  the restriction of $\widehat{\Fun}_V$ to the full subcategory of Artinian objects in $\widehat{\Ca}$. Following \cite[\S 2.6]{sch}, we call the set $\Fun_V(\k[\epsilon])$ the {\it tangent space} of $\Fun_V$, which has a structure of a $\k$-vector space by \cite[Lemma 2.10]{sch}. 
It was proved in \cite[Prop. 2.1]{blehervelez} that $\Fun_V$ satisfies the Schlessinger's criteria \cite[Thm. 2.11]{sch}, that there exists an isomorphism of $\k$-vector spaces 
\begin{equation}\label{hoch}
\Fun_V(\k[\epsilon])\to \Ext_\A^1(V,V),
\end{equation}
and that $\widehat{\Fun}_V$ is continuous in the sense of \cite[\S 14]{mazur}, i.e. for all objects $R$ in $\widehat{\Ca}$, we have  
\begin{equation}\label{cont}
\widehat{\Fun}_V(R)=\invlim_n \Fun_V(R/\m_R^n), 
\end{equation}
where $\m_R$ denotes the unique maximal ideal of $R$. Consequently, $V$ has always a well-defined versal deformation ring $R(\A,V)$ which is also universal provided that $\End_\A(V)$ is isomorphic to $\k$. It was also proved in \cite[Prop. 2.5]{blehervelez} that versal deformation rings are invariant under Morita equivalences between finite dimensional $\k$-algebras. As mentioned in \S \ref{sec1}, if $V$ is a non-projective indecomposable Gorenstein-projective right $\A$-module whose stable endomorphism ring $\SEnd_\A(V)$ is isomorphic to $\k$, then it follows from \cite[Thm. 1.2 (ii) ]{bekkert-giraldo-velez} and \cite[Thm. 1.2 (i)]{velez4} that $R(\A,V)$ is also universal and stable under syzygies. Recall that $\A$ is a Gorenstein $\k$-algebra provided that $\A$ has finite injective dimension as a $\A$-module at both sides. If $\A$ is a self-injective $\k$-algebra then $\A$ is also  Gorenstein and every right $\A$-module is Gorenstein-projective. Therefore the results in \cite[Thm. 1.2 (ii) ]{bekkert-giraldo-velez} and \cite[Thm. 1.2 (i)]{velez4} are applicable in this situation.  

\begin{remark}\label{rem0}
\begin{enumerate}
\item Assume that $\dim_\k \Ext_\A^1(V,V)=r$, then it follows from the isomorphism of $\k$-vector spaces (\ref{hoch}) that the versal deformation ring $R(\A,V)$ is isomorphic to a quotient algebra of the power series ring  $\k[\![t_1,\ldots,t_r]\!]$ and $r$ is minimal with respect to this property. In particular, if $V$ is a right $\A$-module such that $\Ext_\A^1(V,V)=0$, then $R(\A,V)$ is universal and isomorphic to $\k$ (see \cite[Remark 2.1]{bleher15} for more details).
\item In order to prove the main result in this note, we need the following definition and property of morphisms between objects in $\widehat{\Ca}$. Following \cite{sch}, if $R$ is an object in $\widehat{\Ca}$, we denote by $t^\ast_R$ the quotient $\m_R/\m_R^2$ and call it the  {\it Zariski cotangent space} of $R$ over $\k$. Let $\theta: R\to R'$ be a morphism in $\widehat{\Ca}$.  It follows by \cite[Lemma 1.1]{sch} that $\theta$ is surjective if and only if the induced map of cotangent spaces $\theta^\ast: t^\ast_R\to t^\ast_{R'}$ is surjective. 
\end{enumerate}
\end{remark}

\subsection{Proof of Theorem \ref{thm1}}
In the following we assume the hypotheses of Theorem \ref{thm1}. Let $\mathfrak{T}$ be the stable homogeneous tube that as in Figure \ref{fig1}. Let $V = V[1]$ and $0 = V[0]$. For all $\ell \geq 1$, let $\iota_\ell: V[\ell]\to V[\ell+1]$ be the irreducible morphism that is injective and $\pi_\ell: V[\ell + 1] \to V[\ell]$ be the irreducible morphism that is surjective. In particular, for all $\ell \geq 1$, we obtain an almost split sequence
\begin{equation}
0\to V[\ell]\xrightarrow{\begin{pmatrix} \iota_\ell\\ \pi_{\ell-1}\end{pmatrix}}V[\ell+1]\oplus V[\ell-1] \xrightarrow{\begin{pmatrix} \pi_\ell &-\iota_{\ell-1}\end{pmatrix}}V[\ell]\to 0.
\end{equation}  

\begin{claim}
The versal deformation ring $R(\A,V)$ is a quotient of $\k[\![t]\!]$. 
\end{claim}
\begin{proof}
Since we have an almost split sequence $\delta: 0\to V\to V[2]\to V\to 0$, it follows that $\dim_\k\Ext^1_\A(V,V)\geq 1$. Next assume that $\delta': 0\to V\to V'\to V\to 0$ is a non-splitting short exact sequence. Since $\SEnd_\A(V)$ is a division ring by hypothesis, it follows by \cite[Cor. V.2.4]{auslander} that $\delta'$ is also an almost split sequence. Therefore by \cite[Thm. V.1.16]{auslander}, it follows that the short exact sequences $\delta$ and $\delta'$ are equivalent. This proves that $\dim_\k\Ext^1_\A(V,V) =1$ and thus by Remark \ref{rem0}, we obtain that $R(\A,V)$ is a quotient of $\k[\![t]\!]$. 
\end{proof}

\begin{claim}
For each $\ell \geq 2$, the right $\A$-module $V[\ell]$ defines a lift of $V$ over $\k[\![t]\!]/(t^{\ell})$. 
\end{claim}

\begin{proof}
Let $\ell \geq 2$ be fixed but arbitrary. For all $x\in V[\ell]$ let $t$ act on $x$ as $t\cdot x = (\iota_{\ell-1} \circ \pi_{\ell-1})(x)$. Thus, $tV[\ell] \cong V[\ell-1]$, $t^{\ell-1}V[\ell]\cong V[1]$ and $t^\ell V[\ell] = 0$. In this way, $V[\ell]$ becomes a right $\k[\![t]\!]/(t^{\ell})\otimes_\k\A$-module and by using the short exact sequence (\ref{seslong}),we obtain an isomorphism of right $\A$-modules $V[\ell]/tV[\ell]\cong V[1]$. Assume that $d=\dim_\k V$ and let $\{\bar{r}_1,\ldots,\bar{r}_d\}$ be a fixed basis of $V$ over $\k$. Then we can lift the elements $\bar{r}_1,\ldots,\bar{r}_d$ to corresponding elements $r_1,\ldots,r_d\in V[\ell]$ that are linearly independent over $\k$ and such that $\{t^sr_1,\ldots,t^sr_d: 1\leq s\leq \ell-1\}$ is a $\k$-basis of $tV[\ell]\cong V[\ell-1]$, which implies that $\{r_1,\ldots,r_d\}$ is a generating set of $V[\ell]$ as a right $\k[\![t]\!]/(t^{\ell})$-module, i.e. $V[\ell]$ is free over $\k[\![t]\!]/(t^{\ell})$. By using the identification $\k\otimes_{\k[\![t]\!]/(t^{\ell})} V[\ell]\cong V[\ell]/tV[\ell]$, we obtain that there exists an isomorphism of $\A$-modules $\phi_\ell:\k\otimes_{\k[\![t]\!]/(t^{\ell})}V[\ell]\to V[1]$, which implies that $V[\ell]$ induces a lift of $V[\ell]$ over $\k[\![t]\!]/(t^{\ell})$.
\end{proof}

\begin{claim}\label{claim3}
The versal deformation ring $R(\A,V)$ is isomorphic to $\k[\![t]\!]$. 
\end{claim}

\begin{proof}
Observe that the collection of right $\A$-modules $\{V[\ell]\}_{\ell \geq 1}$ with the surjective irreducible morphisms $\pi_{\ell-1}: V[\ell]\to V[\ell -1]$  induces an inverse system that satisfies the Mittag-Leffler condition. Thus by using \cite[Prop. III.10.3]{lang} and applying $\invlim_\ell$ to the short exact sequence (\ref{seslong}), we obtain a short exact sequence of right $\A$-modules 
\begin{equation*}
0\to tW\to W\to V[1]\to 0, 
\end{equation*}
where $W=\invlim_\ell V[\ell]$. In particular $W$ is naturally a $\k[\![t]\!]\otimes_\k\A$-module that is free over $\k[\![t]\!]$ such that $\k\otimes_{\k[\![t]\!]}W\cong W/tW\cong V[1]$, i.e,  $W$ is a lift of $V[1]$ over $\k[\![t]\!]$. Therefore, there exists a $\k$-algebra homomorphism $\theta: R(\A,V[1])\to  \k[\![t]\!]$ in $\widehat{\Ca}$ which corresponds to the versal deformation induced by $W$. Note that since $W/t^2W\cong V[2]$ as $\A$-modules, we obtain that $W/t^2W$ defines a non-trivial lift of $V[1]$ over $\k[\![t]\!]/(t^2)$ and thus there exists a unique morphism $\theta':R(\A,V[1])\to \k[\![t]\!]/(t^2)$. Note that since the cotangent space (as Remark \ref{rem0} (ii)) of $\k[\![t]\!]/(t^2)$ is $1$-dimensional over $\k$, it follows that $\theta'$ is also surjective. Moreover, if $\theta'':\k[\![t]\!]\to \k[\![t]\!]/(t^2)$ is the canonical projection, it follows by the uniqueness of $\theta'$ that $\theta'=\theta''\circ \theta$. Thus again by Remark \ref{rem0} (ii), since $(\theta'')^\ast$ is an isomorphism, we obtain that $\theta^\ast$ and thus $\theta$ is surjective. Therefore by using that  $R(\A,V_[1])$ is a quotient of $\k[\![t]\!]$, we conclude that $\theta$ is an isomorphism. This proves Claim \ref{claim3} and finishes the proof of Theorem \ref{thm1}.
\end{proof}

\subsection{Proof of Corollary \ref{cor1}}\label{subsec2}
In the following we assume that $\k$ is algebraically closed. Assume that $\A = \k Q/ \langle \rho\rangle$ is a special biserial algebra and $b = \alpha_1\cdots \alpha_r$ is a band as described in e.g. \cite[\S 2.2]{schroll2}, where for each $1\leq k\leq r$, $\alpha_k$ is either an arrow or the formal inverse of an arrow of $Q$. In particular, for all $1\leq k\leq r$, we let $v_k$ be the vertex in $Q$ where $\alpha_k$ starts. For all $\lambda \in \k^\ast$ and integers $m\geq 1$, the band $\A$-module $V(b, \lambda, m)$ is the representation of the bound quiver  $(Q,\rho)$ defined as follows. For each vertex $v \in Q_0$, $V(n, \lambda, m)_v$ has a copy of $\k^m$ for each $0\leq k\leq r-1$ such that $v_k = v$. For each $1\leq k \leq r-1$, the linear map $V(b, \lambda, m)_{\alpha_k}$ sends the copy of $\k^m$ associated with the vertex where $\alpha_k$ starts to that associated to where $\alpha_k$ ends via the identity matrix of order $m$. Finally, the linear map $V(b,\lambda, m)_{\alpha_r}$ sends the copy of $\k^m$ associated to the vertex where $\alpha_r$ to that of where $\alpha_r$ ends via the matrix $J_m(\lambda)$ provided that $\alpha_r$ is an arrow, or via the matrix $J_m(\lambda^{-1})$ if $\alpha_r$ is the formal inverse of an arrow, where $J_m(\lambda)$ denotes the $m\times m$ Jordan block with eigenvalue $\lambda$. It follows from e.g. \cite[\S II.4]{erdmann} the band modules lie in homogeneous tubes and in particular the band $\Lambda$-modules $V(b, \lambda, 1)$ lie in the mouth of its homogeneous tube. The following result is straightforward but we decided to include it for completeness. 

\begin{lemma}\label{lem1}
For all integers $m\geq 1$, there exists a short exact sequence of $\A$-modules
\begin{equation*}
    0\to V(b,\lambda, m-1)\xrightarrow{g} V(b,\lambda, m)\xrightarrow{f} V(b,\lambda,1)\to 0.
\end{equation*}
\end{lemma}

\begin{proof}
We define $f: V(b, \lambda, m)\to V(b, \lambda, 1)$ as follows. For each vertex $v \in Q_0$, we let $f_v: V(b, \lambda, m)_v\to V(b, \lambda, 1)_{v}$ be the linear map represented by the $1\times m$-matrix $[1,0,\ldots,0]$. In this way $f$ is surjective and such that $\ker f \cong V(b, \lambda, m-1)$. This proves Lemma \ref{lem1}.
\end{proof}

Thus, Corollary \ref{cor1} is a direct consequence of Theorem \ref{thm1} and Lemma \ref{lem1}.

\subsection{Proof of Corollary \ref{cor2}} 
According to \cite[Thm. 5.1]{schtreval}, there exists a band right $\Lambda$-module $V:=V(b,\beta, 1)$ (for some band $b$) that is a brick, meaning $\End_\A(V) \cong \k$. Consequently, for every $\lambda \in \k^\ast$,  $V(b,\lambda, 1)$ is also a brick. Since these band $\Lambda$-modules $V(b, \lambda, 1)$ lie in the mouth of their homogeneous tube (see, e.g., \cite[\S II.4]{erdmann}), the conclusion directly follows from Corollary \ref{cor1}.


\section{Funding}
The second author was partially supported by CODI (Universidad de Antioquia, UdeA) by project numbers 2022-52654, 2023-62291 and 2024-76672. The third author was supported by the research group CIMI in the Centro de Investigaciones at the Fundaci\'on Universitaria Kornad Lorenz, and by the Office of Academic Affairs at the Valdosta State University. 
\bibliographystyle{amsplain}
\bibliography{UniversalDefRingsTubes}   

\end{document}